\documentclass{amsart}

\usepackage[utf8]{inputenc}
\usepackage[utf8]{inputenc}   
\usepackage[T1]{fontenc}      
\usepackage{lmodern}

\usepackage{amssymb,amsmath,amsfonts,latexsym}
\usepackage{graphicx}
\usepackage[all]{xy}
\usepackage{kmath}
\usepackage{cite}
\usepackage[T1]{fontenc}
\usepackage{hyperref}
\usepackage[parfill]{parskip}
\usepackage{float}

\usepackage{tikz}
\usepackage{newtxtext,newtxmath}
\usepackage{tkz-euclide}
\usetikzlibrary{decorations.markings,arrows}
\tikzset{
    arrowMe/.style={
        postaction=decorate,
        decoration={
            markings,
            mark=at position .5 with {\arrow[thick]{#1}}
        }
    }
}

\usetikzlibrary{backgrounds,calc}
\tikzset{point/.style = {fill=black,circle,inner sep=0.7pt}}

\pgfarrowsdeclaredouble{<<s}{>>s}{stealth}{stealth}
\pgfarrowsdeclaretriple{<<<s}{>>>s}{stealth}{stealth}
\usepackage{enumitem}
\tikzstyle{vertex}=[circle,fill=black!25,minimum size=20pt,inner sep=0pt]
\tikzstyle{edge} = [draw]
\tikzstyle{weight} = [font=\tiny]

\newcommand{\R}{\ensuremath{\mathbb{R}}}
\newcommand{\Aff}{\ensuremath{\text{Aff}}}
\title{Construction of Exponential Families from Statistical Manifolds}

\author{Emmanuel Gnandi}

\address{
INSA de Toulouse, 
 	Département de G\'enie Mathématique, 
 	Université de Toulouse,
 	135 avenue de Rangueil
 	31077 Toulouse cedex 4
 	 France}

\email{kpanteemmanuel@gmail.com, gnandi@insa-toulouse.fr}

\newtheorem{thm}{Theorem}[section]

\newtheorem{prop}{Proposition}[section]

\newtheorem{rem}[thm]{Remark}
\newtheorem{defn}[thm]{Definition}

\numberwithin{equation}{section}
\newtheorem{disc}[thm]{Discussion}

\begin{document}

\begin{abstract}
We investigate the construction of exponential families from statistical manifolds, a central problem in information geometry. We prove that every compact statistical manifold admits a singular foliation whose leaves are Hessian manifolds. In particular, any non-flat, compact, orientable $3$-dimensional leaf arises as a quotient of an exponential family and has only odd Betti numbers. Our approach is constructive: we explicitly describe the foliation and analyze the geometric and topological properties of its leaves. We show that compact orientable leaves are either finite quotients of flat torus or mapping torus with periodic monodromy. In three dimensions, non-flat leaves admit a co-Kähler structure, which allows us to realize them as explicit exponential families parametrized by a Lorentz cone. These results establish a concrete bridge between abstract statistical manifolds and exponential families, highlighting deep connections between information geometry, differential geometry, and the topology of $3$-manifolds.
\end{abstract}

\maketitle

\section{Introduction}

Information geometry is a rapidly expanding field that reveals a deep interplay between differential geometry and statistical theory through the Fisher information metric (see \cite{amari2000methods, molitor2013exponential, amari2016information, lauritzen1987statistical, ay2017information}). Considerable attention has been devoted to the problem of defining meaningful distances between probability distributions, given their fundamental importance in a wide range of applications such as theoretical statistics, information theory, signal processing, quantitative finance, and machine learning. In this context, Rao~\cite{rao1945information} introduced the geodesic distance induced by the Fisher information metric, which is now commonly referred to as the Fisher--Rao distance (or Rao’s distance). The Fisher metric, also known as the information metric, is in general positive semi-definite, although it may be degenerate. \v{C}encov proved a fundamental uniqueness theorem for the Fisher information metric on finite sample spaces (\cite{cencov2000statistical}, p.~156). This result was later extended to infinite sample spaces in \cite{ay2015information}.  

The Fisher information metric is invariant under reparametrizations of the sample space $\Xi$ and covariant with respect to reparametrizations of the parameter space $\Theta$ (see \cite{calin2014geometric}). Owing to these properties, the Fisher metric has found widespread applications in various branches of geometry (see \cite{hitchin2006geometry,gnandi2022dual, fujiwara2020dually, groisser1997instantons, ciaglia2021focus, murray1993information, matsueda2017banados, naudts2009q, petz2007quantum, lesniewski1999monotone, facchi2010classical, barbaresco2014koszul}).  

Information geometry is also deeply connected with several mathematical frameworks, including Hessian geometry, Kähler geometry, complex geometry, special Kähler geometry, and complex analysis. One of the central problems in this area concerns the realization of a statistical manifold as a statistical model, namely, a manifold endowed with dual affine connections compatible with a Fisher metric. This problem was resolved by Hông Vân Lê in \cite{le2006statistical}. A further and significantly more difficult question is to characterize which classes of statistical manifolds can be realized as exponential families; this problem remains open. Despite its rapid progress, information geometry continues to present a number of important open questions (see \cite{sei2025open}).  

The remainder of this paper is organized as follows. Section~\ref{thm:Basic} recalls the necessary background material in information geometry. Section~\ref{thm:Result} is devoted to the proof of the main result.

\section{Basic Notions}
\label{thm:Basic}

This section introduces the fundamental structures studied throughout this work. For a comprehensive background and a collection of illustrative examples, we refer the reader to the following references:
\cite{amari2016information, amari2000methods, amari2012differential, amari2014curvature, ay2017information, calin2014geometric, molitor2013exponential, mehrshad2025statsitical, gnandi2022classification,gnandi2022dual,furuhata2017sasakian, molitor2025kahler, noguchi1992geometry, pal2024introduction, lauritzen1987statistical, cencov2000statistical, ay2002dually, zhang2007note, furuhata2016submanifold, matsuzoe2010statistical, takano2006statistical}.

\subsection{Statistical models}

\begin{defn}\cite{amari1987differential, amari2000methods, amari2012differential, amari2014curvature, amari2016information, ay2017information, calin2014geometric}
Let $\mathcal{S}$ be a family of probability distributions represented as
\[
\mathcal{S} = \{p(x, \theta) \mid \theta \in \Theta\},
\]
where $\Theta$ is a domain in $\mathbb{R}^n$ with Cartesian coordinates $(\theta^1, \theta^2, \dots, \theta^n)$, called the parameter space. Assume further that the mapping $\theta \mapsto p(\cdot, \theta)$ is bijective. Then $\mathcal{S}$ is called a statistical model determined by $p$.
\end{defn}
The Fisher information metric \(g^F\) on \(\mathcal{S}\) is defined in local coordinates \(\theta = (\theta^1, \dots, \theta^n)\) by \[ g^F_{ij}(\theta) = \mathbb{E}_{p(\cdot, \theta)} \left[ \frac{\partial \log p(x, \theta)}{\partial \theta^i} \, \frac{\partial \log p(x, \theta)}{\partial \theta^j} \right], \] where \(\mathbb{E}_{p(\cdot, \theta)}\) denotes expectation with respect to \(p(\cdot, \theta)\). We assume: \begin{enumerate}[label=(\roman*)] \item \(g^F_{ij}(\theta)\) is finite and smooth, \item the matrix \(\big( g^F_{ij}(\theta) \big)\) is positive definite for all \(\theta \in \Theta\). \end{enumerate} The corresponding Riemannian metric is \[ g^F = \sum_{i,j=1}^n g^F_{ij}(\theta) \, d\theta^i \otimes d\theta^j, \] defines a Riemannian metric on \(\mathcal{S}\), called the Fisher metric of \(\mathcal{S}\).\\ For $\alpha \in \mathbb{R}$, the $\alpha$-connection $\nabla^{(\alpha)}$ on a statistical model $\mathcal{S}$ is defined by \[ \begin{aligned} \Gamma^{(\alpha)}_{ijk}(\theta) &= \mathbb{E}_{p(\cdot, \theta)} \Bigg[ \Bigg( \frac{\partial^2 \log p(x, \theta)}{\partial \theta^i \partial \theta^j} + \frac{1 - \alpha}{2} \frac{\partial \log p(x, \theta)}{\partial \theta^i} \frac{\partial \log p(x, \theta)}{\partial \theta^j} \Bigg) \\ &\qquad\qquad \cdot \log p(x,\theta)\, \frac{\partial \log p(x, \theta)}{\partial \theta^k} \Bigg]. \end{aligned} \] In the literature, the family of $\alpha$-connections $\nabla^{(\alpha)}$ is commonly referred to as the Amari-Chentsov $\alpha$-connections or Amari--\v{C}encov $\alpha$-connections \cite{amari2000methods} . This terminology honors the foundational contributions of Shun-ichi Amari \cite{amari1987differential,amari2012differential} and Nikolai \v{C}encov\cite{cencov2000statistical} to information geometry. The Amari-Chentsov $\alpha$-connections provide a unified framework for studying the differential geometric structure of statistical models, with the Fisher metric as the central Riemannian structure and the $\alpha$-connections forming a dualistic structure on the statistical manifold. The $\alpha$-connection $\nabla^{(\alpha)}$ satisfies the following fundamental properties: \begin{enumerate} \item $\nabla^{(\alpha)}$ is torsion-free for all $\alpha \in \mathbb{R}$. \item $\nabla^{(0)}$ coincides with the \emph{Levi-Civita connection} of the Fisher metric. \end{enumerate} Among all $\alpha$-connections, the $\pm 1$-connections play a particularly important role in the geometric theory of statistical inference. $\nabla^{(1)}$ is called the \emph{exponential connection} (or \emph{e-connection}). $\nabla^{(-1)}$ is called the \emph{mixture connection} (or \emph{m-connection}). \vspace{0.2cm} It's well known that $\alpha$-connections, the following duality formula holds: \[ \partial_{\theta_k} g^F(\partial_{\theta_i}, \partial_{\theta_j}) = g^F\!\left( \nabla^{(\alpha)}_{\partial_{\theta_k}} \partial_{\theta_i},\, \partial_{\theta_j} \right) \;+\; g^F\!\left( \partial_{\theta_i},\, \nabla^{(-\alpha)}_{\partial_{\theta_k}} \partial_{\theta_j} \right). \] The connections $\nabla^{(\alpha)}$ and $\nabla^{(-\alpha)}$ are said to be dual (or conjugate) with respect to $g^F$. \vspace{0.2cm} For any $\alpha, \gamma \in \mathbb{R}$, the Christoffel symbols of the Amari-Chentsov connections $\nabla^{(\alpha)}$ and $\nabla^{(\gamma)}$ are related by: \[ \Gamma^{(\gamma)}_{ij,k} = \Gamma^{(\alpha)}_{ij,k} + \frac{\alpha - \gamma}{2} \, \text{T}^F_{ijk}, \] where \[ \text{T}^F_{ijk}(\theta) = \mathbb{E}_{p(\cdot, \theta)}\left[ \partial_i \ell_\theta \, \partial_j \ell_\theta \, \partial_k \ell_\theta \right], \quad \text{with } \ell_\theta(x) = \log p(x, \theta). \] The fully symmetric $(0,3)$-tensor field $\text{T}^F$ defined by the components $\text{T}^F_{ijk}$ is called the cubic form or Amari-Chentsov tensor. Moreover, the covariant derivative of the Fisher metric with respect to $\nabla^{(\alpha)}$ satisfies: \[ (\nabla^{(\alpha)}_X g^F)(Y, Z) = \alpha \, \text{T}^F(X, Y, Z). \] An important and widely studied class of statistical models is the exponential family, consisting of probability densities of the form \[ p(x, \theta) = \exp\left[ C(x) + \sum_{i=1}^n \theta^i F_i(x) - \psi(\theta) \right], \quad \theta \in \Theta, \] where $C(x), F_1(x), \dots, F_n(x)$ are smooth functions of the observation $x$, and $\psi$ is the cumulant generating function on the parameter space $\Theta$. The coordinates $\theta = (\theta^1, \dots, \theta^n)$ are called natural parameters, and $\psi$ is determined by the normalization condition: \[ \psi(\theta) = \log \int \exp\left[ C(x) + \sum_{i=1}^n \theta^i F_i(x) \right] dx. \] A particularly important special case is the normal distribution, obtained by taking \[ C(x) = 0, \quad F_1(x) = x, \quad F_2(x) = x^2, \quad \theta^1 = \frac{\mu}{\sigma^2}, \quad \theta^2 = -\frac{1}{2\sigma^2}, \] which yields \[ \psi(\theta^1, \theta^2) = -\frac{(\theta^1)^2}{4\theta^2} + \frac{1}{2} \log\left( -\frac{\pi}{\theta^2} \right). \] This corresponds exactly to the Gaussian distribution $\mathcal{N}(\mu, \sigma)$. For an exponential family with natural parameters $\theta = (\theta^1, \dots, \theta^n)$, the Fisher metric and the cubic form are given in terms of the cumulant generating function $\psi(\theta)$ by: \[ g^F_{ij}(\theta) = \partial_i \partial_j \psi(\theta), \quad \text{T}^F_{ijk}(\theta) = \partial_i \partial_j \partial_k \psi(\theta), \] where $\partial_i = \frac{\partial}{\partial \theta^i}$. \subsection{Statistical manifold} Building upon the geometric insights developed by Amari and Chentsov, Steffen L. Lauritzen provided a rigorous formalization of the notion of a statistical manifold. \begin{defn}[Lauritzen \cite{lauritzen1987statistical}] A statistical manifold is a quadruple $(M, g, \nabla, \nabla^*)$ where: \begin{enumerate} \item For all $X, Y, Z \in \mathfrak{X}(M)$, \[ X(g(Y, Z)) = g(\nabla_X Y, Z) + g(Y, \nabla^*_X Z). \] \item Both $\nabla$ and $\nabla^*$ are torsion-free affine connections. \end{enumerate} Equivalently, a statistical manifold may be defined as a triplet $(M, g,\nabla)$ such that $\text{T}= \nabla g$ is a totally symmetric $(0,3)$-tensor field. \end{defn} For a statistical manifold $(M, g, \nabla, \nabla^{*})$, the triple $(g, \nabla, \nabla^{*})$ is called a statistical structure, and the pair $(\nabla, \nabla^{*})$ are referred to as statistical connections. The Levi--Civita connection $\nabla^{\mathrm{LC}}$ satisfies \begin{equation} 2\, \nabla^{\mathrm{LC}} = \nabla + \nabla^{*}. \end{equation} Note that the $(0,3)$-tensor field $\text{T}$ is referred to by Lauritzen \cite{lauritzen1987statistical} as the skewness tensor.\\ The following construction provides a fundamental example of how statistical structures naturally arise on Riemannian manifolds equipped with a distinguished vector field. Let $(M, g)$ be a Riemannian manifold with Levi–Civita connection $\nabla^{\mathrm{LC}}$, and let $\xi \in TM \setminus \{0\}$ be a non-vanishing vector field. For all $X, Y \in \mathfrak{X}(M)$, define 
\begin{align} \nabla_X Y &= \nabla^{\mathrm{LC}}_X Y + g(X, \xi) g(Y, \xi) \, \xi, \label{eq:nabla-def} \\ \nabla^*_X Y &= \nabla^{\mathrm{LC}}_X Y - g(X, \xi) g(Y, \xi) \, \xi.
\end{align} 

Then $(M, g, \nabla, \nabla^*)$ is a statistical manifold. Starting from our previous construction, on an almost contact metric manifold \((M^{2d+1},\phi,g,\alpha,\xi)\) with \(\nabla^{LC}\) the Levi--Civita connection of the metric \(g\) (see \cite{blair2010riemannian} for exposition on almost contact metric geometry), we can construct the following one-parameter families of statistical connections: 
\begin{align} \nabla^{\varepsilon} &= \nabla^{LC} + \varepsilon\,\alpha(\cdot)\alpha(\cdot)\xi, \qquad \forall \varepsilon\in\mathbb{R},\\ \nabla^{-\varepsilon} &= \nabla^{LC} - \varepsilon\,\alpha(\cdot)\alpha(\cdot)\xi, \qquad \forall \varepsilon\in\mathbb{R}. 
\end{align}\\For a statistical manifold $(M, g, \nabla, \nabla^*)$, the connection $\nabla$ is flat if and only if $\nabla^*$ is flat. In this case, we say that $(M, g, \nabla, \nabla^*)$ is a dually flat manifold(\cite{amari2016information}). It was shown in \cite{ay2002dually} that any compact dually flat manifold necessarily has an infinite fundamental group. Consequently, the sphere \(\mathbb{S}^{q}\), for \(q > 1\), cannot be a dually flat manifold. \subsection{Hessian manifold} Hessian manifolds exhibit particularly rich topological and geometric structures and share a deep relationship with Kähler manifolds. This relationship is especially manifested in the fact that the tangent bundle of any Hessian manifold is naturally endowed with a Kähler metric induced by the Hessian structure itself~\cite{shima2007geometry,dombrowski1962geometry,satoh2019almost}. Such manifolds occupy a central place in information geometry, Kähler geometry, special Kähler geometry, and complex analysis. They also play an important role in various applied fields, including convex optimization, statistics, combinatorics, Souriau's thermodynamics on Lie groups\cite{souriau1997,barbaresco2015koszul}. \begin{defn} A locally flat manifold is a pair $(M, \nabla)$, where $\nabla$ is a connection whose torsion and curvature tensors vanish identically: \[ T^\nabla = 0, \qquad R^\nabla = 0. \] Such a connection $\nabla$ is called a locally flat connection, and $(M, \nabla)$ is referred to as a locally flat manifold. \end{defn} \begin{defn} \label{def:affine_structure} An affinely flat structure on an $m$-dimensional manifold $M$ is a complete atlas \begin{equation} \mathcal{A} = \{ (U_i,\Phi_i) \}, \end{equation} such that the transition maps $\Phi_i^{-1} \circ \Phi_j$ belong to the group $\Aff(\R^n)$ of affine transformations on $\mathbb{R}^m$. \end{defn} It's important to note that a manifold admits an affine structure if and only if it is locally flat. \begin{defn}[\cite{shima1997geometry, shima2007geometry}] A Riemannian metric $g$ on a locally flat manifold $(M, \nabla)$ is called a Hessian metric if, locally, there exists a smooth function $f$ such that \[ g = \nabla^2 f, \quad \text{i.e.,} \quad g_{ij} = \frac{\partial^2 f}{\partial x_i \partial x_j}, \] where $(x_1, \dots, x_n)$ is a system of affine coordinates with respect to $\nabla$. A locally flat manifold $(M, \nabla)$ equipped with a Hessian metric $g$ is called a Hessian manifold and is denoted by $(M, g, \nabla)$. \end{defn} Let $D$ denote the Levi-Civita connection of $(M,g)$, and define \[ \gamma = D - \nabla. \] Since both $\nabla$ and $D$ are torsion-free, we have \[ \gamma_X Y = \gamma_Y X, \quad \text{for all vector fields } X,Y. \] Moreover, in affine coordinates, the components $\gamma^{i}_{jk}$ of $\gamma$ coincide with the Christoffel symbols $\Gamma^{i}_{jk}$ of $D$. \begin{prop}[\cite{shima1997geometry}] Let $(M, \nabla)$ be an locally flat manifold and $g$ a Riemannian metric. The following are equivalent: \begin{enumerate} \item $g$ is a Hessian metric; \item $(\nabla_X g)(Y,Z) = (\nabla_Y g)(X,Z)$ for all vector fields $X,Y,Z$; \item In affine coordinates $(x^i)$, $g_{ij}$ satisfies \[ \frac{\partial g_{ij}}{\partial x^k} = \frac{\partial g_{kj}}{\partial x^i}; \] \item $g(\gamma_X Y, Z) = g(Y, \gamma_X Z)$ for all vector fields $X,Y,Z$; \item $\gamma_{ijk} = \gamma_{jik}$. \end{enumerate} \end{prop} On an orientable Hessian manifold $(M,g,\nabla)$, using the connection $\nabla$ and the volume element induced by $g$, Koszul introduced a closed 1-form, which Shima~\cite{shima2007geometry} refers to as the first Koszul form, associated with the Hessian structure $(\nabla, g)$. \begin{defn}[\cite{koszul1961domaines, shima1997geometry}] The first Koszul form $\beta$ associated with $(\nabla, g)$ is the $1$-form defined by \[ \nabla_X (\mathrm{Vol}_g) = \beta(X) \, \mathrm{Vol}_g, \quad X \in \mathfrak{X}(M), \] where $\mathrm{Vol}_g$ denotes the Riemannian volume form induced by $g$. \end{defn} Directly from the definition, \[ \beta(X) = \operatorname{Tr}(\gamma_X), \] and in affine coordinates, \[ \beta_i = \frac{1}{2} \frac{\partial \log \det(g_{pq})}{\partial x_i} = \gamma^k_{ki}. \] We are now able to present the main theorem. 
\section{Main Results}
\label{thm:Result}
Obviously, a statistical model is a statistical manifold. Lauritzen and Amari raised the following fundamental question: can every statistical manifold be realized as a finite-dimensional statistical model ? Constructing a realization of a statistical manifold as a statistical model, i.e., a manifold with dual connections with respect to a Fisher metric, is an important question in information geometry. While a positive answer to this problem was given by the work of Hông Vân Lê~\cite{le2006statistical}, explicitly constructing the probability family that gives rise to the Fisher metric remains a challenging task in general.

\begin{thm}\cite{le2006statistical}
Every smooth compact statistical manifold $(M^n, g, T)$ of finite dimension  
admits, for a sufficiently large integer $N$, an isostatistical embedding  
(i.e., an embedding that preserves both the tensor $T$ and the metric $g$) into $\mathrm{Cap}^{N}_{+}$, the space of all positive probability distributions defined on a discrete sample space $\Omega^{N}$ consisting of $N$-elementary events, where
\[
\mathrm{Cap}^{N}_{+} := \left\{ (p_1, \ldots, p_n) \;\middle|\; p_i > 0 \text{ for } i = 1, \ldots, n
\;\text{and}\; \sum_{i=1}^n p_i = 1 \right\}.
\]
\end{thm}

This result shows that compact statistical manifolds of finite dimension 
can indeed be realized as statistical models, although it does not provide 
an explicit form for such an embedding.

A major challenge in information geometry is to identify which classes of statistical manifolds $(M, g, T)$ arise from an exponential family. 
To date, to the best of our knowledge, this problem remains unsolved.
\vspace{0.2cm}

We can now state the main result of this paper.
\begin{thm}
\label{thm:statistical_foliation}
On any compact statistical manifold, there exists a singular foliation whose leaves are Hessian manifolds. Every compact, orientable leaf is either a finite quotient of a flat torus or a mapping torus with a periodic gluing map. If such a leaf is three-dimensional and does not admit any flat metric, then all of its Betti numbers are odd, and it arises as a quotient of an exponential family.
\end{thm}

\begin{proof}
The proof proceeds in four steps. The first step is presented in the PhD thesis(\cite{gnandi2022classification}), but we have deemed it useful to provide additional details to facilitate understanding for the readers of this paper.

\begin{enumerate}[leftmargin=*,label=\textbf{Step \arabic*}:]
    \item \textbf{Foliations by Hessian manifolds.}
    
    Let $(M, g, \nabla)$ be a statistical manifold. The equation associated with $\nabla$ is defined by:
    \[
    \mathrm{E}(\nabla): \quad \nabla^2 X = 0 \quad \text{for all } X \in \mathfrak{X}(M),
    \label{eq:hessian_eq}
    \]
    where $\nabla^2$ denotes the second covariant derivative, given for all $X, Y, Z \in \mathfrak{X}(M)$ by:
    \[
    \nabla^2_{X,Y} Z := \nabla_X(\nabla_Y Z) - \nabla_{\nabla_X Y} Z.
    \]
    Let $(x^1, \dots, x^m)$ be local coordinates on $M$ and $X \in \mathfrak{X}(M)$. Setting
    \begin{align*}
    \partial_i &= \frac{\partial}{\partial x^i}, & 
    X &= \sum_k X^k \partial_k, & 
    \nabla_{\partial_i} \partial_j &= \sum_k \Gamma^k_{ij} \partial_k,
    \end{align*}
    the princip al symbol of the operator $X \mapsto \nabla^2 X$ takes the form:
    \[
    (\nabla^2 X)(\partial_i, \partial_j) = \sum_k \theta^k_{ij}(X) \partial_k,
    \]
    where
    \begin{align*}
    \theta^l_{ij}(X) &= \frac{\partial^2 X^l}{\partial x^i \partial x^j} + \Gamma^l_{ik} \frac{\partial X^k}{\partial x^j} + \Gamma^l_{jk} \frac{\partial X^k}{\partial x^i} - \Gamma^k_{ij} \frac{\partial X^l}{\partial x^k} \\
    &\quad + \frac{\partial \Gamma^l_{jk}}{\partial x^i} + \Gamma^m_{jk} \Gamma^l_{im} - \Gamma^m_{ij} \Gamma^l_{mk}.
    \end{align*}
    The equation $\mathrm{E}(\nabla)$ thus reduces locally to the system $\theta^l_{ij}(X) = 0$.

    Let $\mathscr{J}_\nabla(M)$ denote the sheaf of solutions to the  equation $\mathrm{E}(\nabla)$ and $J_\nabla = \Gamma(M, \mathscr{J}_\nabla)$ its space of global sections. Define a product on $J_\nabla$ by:
    \[
    X \cdot Y = \nabla_X Y.
    \]
    A direct computation shows:
    \[
    X \cdot (Y \cdot W) - (X \cdot Y) \cdot W = \nabla_X(\nabla_Y W) - \nabla_{\nabla_X Y} W = \nabla^2_{X,Y} W.
    \]
    Hence, for all $W \in J_\nabla$,
    \[
    X \cdot (Y \cdot W) - (X \cdot Y) \cdot W = 0.
    \]
    Now, for $U, V \in J_\nabla$, we have:
    \begin{align*}
    X \cdot (Y \cdot (U \cdot V)) - (X \cdot Y) \cdot (U \cdot V) 
    &= X \cdot ((Y \cdot U) \cdot V) - ((X \cdot Y) \cdot U) \cdot V \\
    &= [X \cdot (Y \cdot U) - (X \cdot Y) \cdot U] \cdot V = 0.
    \end{align*}
    Moreover, $\nabla_U V \in J_\nabla$, since:
    \[
    \nabla^2_{X,Y} (\nabla_U V) = X \cdot (Y \cdot (U \cdot V)) - (X \cdot Y) \cdot (U \cdot V) = 0.
    \]
    Thus, $(J_\nabla, \cdot)$ is an associative algebra. The principal symbol analysis shows that $\mathrm{E}(\nabla)$ is of finite type \cite{guillemin1965integrability,singer1965infinite,kumpera1972lie}, so $J_\nabla$ is finite-dimensional. As $\nabla$ is torsion-free, the bracket
    \[
    [X, Y]_\nabla = \nabla_X Y - \nabla_Y X
    \]
    coincides with the Lie bracket, and $J_\nabla$ is closed under it, forming a finite-dimensional Lie subalgebra of $\mathfrak{X}(M)$. By Lie's third theorem, there exists a unique simply connected Lie group $G_\nabla$ with Lie algebra isomorphic to $J_\nabla$. Since $M$ is compact, $J_\nabla$ is integrable \cite{palais1957global}, yielding a locally effective differentiable action:
    \[
    G_\nabla \times M \to M.
    \]
    
    Let $\mathcal{F}_\nabla$ be the singular foliation defined by the orbits of $G_\nabla$, and $g_\nabla$ the restriction of $g$ to the leaves of $\mathcal{F}_\nabla$. Since $J_\nabla$ is an associative algebra with respect to $\nabla$, each leaf is $\nabla$-autoparallel. Let $\nabla^{\mathcal{F}_\nabla}$ denote the restriction of $\nabla$ to the leaves. A straightforward computation shows that $\nabla^{\mathcal{F}_\nabla}_X g_\nabla$ is totally symmetric. Moreover, for $U \in J_\nabla$, we have
\[
R^{\nabla^{\mathcal{F}_\nabla}}(X,Y)U = \nabla^{\mathcal{F}_\nabla, 2}_{X,Y}U - \nabla^{\mathcal{F}_\nabla, 2}_{Y,X}U = 0 
\quad \text{for all } U \in J_\nabla,
\]
where
\[
\nabla^{\mathcal{F}_\nabla, 2}_{X,Y}Z = \nabla^{\mathcal{F}_\nabla}_X\big(\nabla^{\mathcal{F}_\nabla}_Y Z\big) 
- \nabla^{\mathcal{F}_\nabla}_{\nabla^{\mathcal{F}_\nabla}_X Y} Z.
\]

    Hence, the leaves of $\mathcal{F}_\nabla$ are  dually flat manifolds, and thus Hessian manifolds.
    
    \vspace{0.2cm}
    Then, on any statistical manifold $(M, g, \nabla)$, there always exists a singular foliation 
$\mathcal{F}_\nabla$ whose leaves are Hessian manifolds. A similar construction is available 
for the dual connection $\nabla^*$. This completes the first step of the proof of the theorem.

    \item \textbf{Topology of compact orientable leaves.}
    
    Let $L_\nabla$ be a compact orientable leaf of the singular foliation $\mathcal{F}_\nabla$. By Step~1, the triple $(L_\nabla, g_\nabla, \nabla^{\mathcal{F}_\nabla})$ is a compact orientable Hessian manifold. Let $D$ denote the Levi-Civita connection of $g_\nabla$. The first Koszul form $\beta \in \Omega^1(L_\nabla)$ associated with the statistical structure $(\nabla^{\mathcal{F}_\nabla}, g_\nabla)$ is defined by the relation
    \[
    \nabla^{\mathcal{F}_\nabla}_X (\mathrm{Vol}_{g_\nabla}) = \beta(X) \, \mathrm{Vol}_{g_\nabla}, \quad \text{for all } X \in \mathfrak{X}(L_\nabla),
    \]
    where $\mathrm{Vol}_{g_\nabla}$ is the Riemannian volume form induced by $g_\nabla$ \cite{koszul1961domaines, shima1997geometry}.

    According to Shima and Yagi \cite[Theorem~4.1]{shima1996geometry}, the first Koszul form $\beta$ is $D$-parallel, i.e., $D\beta = 0$. Consequently, its norm $\|\beta\|_{g_\nabla}$ is constant on $L_\nabla$. This leads to two mutually exclusive cases:

    \paragraph{Case 1: Vanishing Koszul form ($\beta = 0$).} 
    By \cite[Theorem~4.2]{shima1996geometry}, the identity $\beta = 0$ implies that the Levi-Civita connection coincides with the flat affine connection: $D = \nabla^{\mathcal{F}_\nabla}$. Hence $(L_\nabla, g_\nabla)$ is a flat Riemannian manifold. Since $L_\nabla$ is compact, the classical theorems of Bieberbach \cite{bieberbach1911bewegungsgruppen,bieberbach1912bewegungsgruppen} and Wolf \cite{Wolf} imply that $L_\nabla$ is isometric to a flat torus:
    \[
    L_\nabla \simeq \mathbb{R}^d / \Lambda,
    \]
    where $d = \dim L_\nabla$ and $\Lambda \subset \mathbb{R}^d$ is a Bieberbach group, acting freely on $\mathbb{R}^{d}$ and $\simeq$
, means “isometric to.” In other words, $L_\nabla$ is finitely covered by a flat torus, such as $\mathbb{T}^d = \mathbb{R}^d / \mathbb{Z}^d$.

    \paragraph{Case 2: Non-vanishing Koszul form ($\beta \neq 0$).} 
    In this case, $\beta$ is a non-vanishing $D$-parallel $1$-form. Define the vector field $Z$ dual to $\beta$ via the metric: $\beta = g_\nabla(Z, \cdot)$. Since $D\beta = 0$, it follows that $DZ = 0$, i.e., $Z$ is a parallel vector field on $L_\nabla$.

    By Welsh's classification \cite[Theorem~1]{welsh1986manifolds}, the existence of a non-trivial parallel vector field on a compact Riemannian manifold implies that $L_\nabla$ is a fiber bundle over $\mathbb{S}^1$ with finite structure group denoted $G$. Furthermore, the results of Tischler \cite{tischler1970fibering} imply that 
$L_\nabla$ is diffeomorphic to a mapping torus:
\[
L_\nabla \cong X_\varphi = \frac{X \times [0,1]}{(p,0) \sim (\varphi(p),1)},
\]
with the associated fibre bundle 
\[
X \rightarrow L_\nabla = X_\varphi \rightarrow \mathbb{S}^{1},
\]
where $\varphi \in \mathrm{Diff}(X)$.
By \cite[Proposition 6.4]{bazzoni2014structure}, the structural group $G$
generates a finite cyclic subgroup $\langle \varphi \rangle \subset \mathrm{Diff}(X)$.
From \cite{welsh1986manifolds}, we deduce that $\varphi$ is of finite order,
and in particular periodic. This concludes the second step.

    \item \textbf{Three-dimensional non-flat leaves.}
    
    Now assume that the leaf $L_\nabla$ is 3-dimensional and non-flat. In this case, the first Koszul form $\beta$ is a non-vanishing closed 1-form. Define the corresponding non-vanishing vector field $Z$ by $\beta = g_\nabla(Z, \cdot)$. Since $D\beta = 0$, it follows that $DZ = 0$, and consequently $g_\nabla(Z,Z) = c > 0$ is constant. 

    Define the normalized vector field and 1-form by:
    \[
    \widetilde{Z} := \frac{Z}{\sqrt{c}}, \qquad \widetilde{\beta} := \frac{\beta}{\sqrt{c}}.
    \]
    Then we have:
    \[
    \widetilde{\beta}(\widetilde{Z}) = 1, \quad D\widetilde{Z} = 0, \quad \mathcal{L}_{\widetilde{Z}} g_\nabla = 0, \quad g_\nabla(\widetilde{Z},\widetilde{Z}) = 1.
    \]

    Define a 2-form by:
    \[
    \Omega = i_{\widetilde{Z}} \mathrm{Vol}_{g_\nabla},
    \]
    where $\mathrm{Vol}_{g_\nabla}$ denotes the Riemannian volume form. A straightforward computation shows:
    \[
    d\Omega = d(i_{\widetilde{Z}} \mathrm{Vol}_{g_\nabla}) = \mathcal{L}_{\widetilde{Z}} \mathrm{Vol}_{g_\nabla} = 0.
    \]

    Observe that:
    \[
    i_{\widetilde{Z}}(\widetilde{\beta} \wedge \mathrm{Vol}_{g_\nabla}) = 0,
    \]
    which immediately yields:
    \[
    \mathrm{Vol}_{g_\nabla} = \widetilde{\beta} \wedge \Omega.
    \]

    We thus deduce that $\widetilde{Z}$ is the Reeb vector field associated with the cosymplectic structure $(\widetilde{\beta},\Omega)$ (see \cite{li2008topology}). Since $\widetilde{Z}$ is a Killing vector field on $L_\nabla$, it follows from~\cite{bazzoni2015k} that $(L_\nabla,\widetilde{\beta},\Omega)$ is a K-cosymplectic manifold. By~\cite[Proposition~2.8]{bazzoni2015k}, $L_\nabla$ admits an K-cosymplectic structure $(\widetilde{\beta},\widetilde{Z},\phi,h)$ with $\mathcal{L}_{\widetilde{Z}} h = 0$. 

    From Goldberg~\cite[Proposition~3]{goldberg1969integrability}, the structure $(\widetilde{\beta},\widetilde{Z},\phi,h)$ is normal (i.e., $N_\phi = 0$); hence $(L_\nabla, \widetilde{\beta}, \widetilde{Z}, \phi, h)$ is a co-Kähler manifold (Kähler-mapping torus \cite{li2008topology}). Furthermore, by \cite[Theorem~11]{Chinea1993TopologyOC}, the first Betti number $b_1(L_\nabla)$ is odd. Applying Poincaré duality, it follows that all Betti numbers $b_s(L_\nabla)$ are odd.

    From \cite{cornu1987varietes} (Proposition 2), $(L_\nabla,\widetilde{\beta},\Omega)$ cannot be a contact manifold. Then from \cite{geiges1997normal} (Theorems 1 and 2), we deduce that $L_\nabla$ is diffeomorphic to one of the following:
    \begin{enumerate}[]
        \item $\Lambda \backslash (\mathbb{H}^2 \times \mathbb{R})$ with $\Lambda \subset \mathrm{Isom}_0(\mathbb{H}^2 \times \mathbb{R})$,
        \item $\mathbb{T}^2$-bundles over $\mathbb{S}^1$ with periodic monodromy,
        \item $\mathbb{S}^2 \times \mathbb{S}^1$.
    \end{enumerate}

    Since $(L_\nabla, g_\nabla, \nabla^{\mathcal{F}_\nabla})$ is a Hessian manifold, Shima~\cite{shima1981hessian} shows that $\mathbb{S}^2 \times \mathbb{S}^1$ does not admit a Hessian metric, so this case is excluded.

    Now assume $L_\nabla$ is a $T^2$-bundle over $\mathbb{S}^1$ with periodic monodromy. From \cite{martelli2016introduction} (Exercise 11.4.13), we deduce that $L_\nabla = \mathbb{T}^2_A$, where $A \in \mathrm{SL}_2(\mathbb{Z})$ with $A = \pm I$ or $|\mathrm{tr}(A)| < 2$. Then from \cite{martelli2016introduction} (Proposition 11.4.14), $L_\nabla$ is a Seifert manifold with Euler number zero and whose base orbifold has zero Euler characteristic. By \cite{martelli2016introduction} (Theorem 12.3.1), $L_\nabla$ admits a flat metric. Since $L_\nabla$ is non-flat by assumption, this case is also excluded.

    Therefore, the only remaining possibility is:
    \[
    L_\nabla \cong \frac{\mathbb{H}^2 \times \mathbb{R}}{\Gamma} \quad \text{with} \quad \Gamma \subset \mathrm{Isom}_0(\mathbb{H}^2 \times \mathbb{R})\cong \mathrm{Isom}_0(\mathbb{H}^2)\times \mathrm{Isom}_0(\mathbb{R}).
    \]

    \item \textbf{Exponential family structure.}

    \vspace{0.2cm}
    Following \cite{molnar1997projective}, we identify the manifold $\mathbb{H}^2 \times \mathbb{R}$ 
      with the Lorentz cone ( the positive light-cone)
    \[
    \Theta = \left\{(x,y,z) \in \mathbb{R}^3 : z > 0,\ z^2 > x^2 + y^2 \right\}.\]
By \cite[Example 4.2]{shima2007geometry}, $\Theta$ is a regular self-dual convex cone. Then for $\theta = (x,y,z) \in \Theta$ and $\theta^{*} = (\xi,\eta,\zeta) \in \Theta$, consider the 
exponential family on $\Theta$ defined by
\[
P((x,y,z),(\xi,\eta,\zeta)) = \frac{e^{-(x\xi + y\eta + z\zeta)}}{\chi(x,y,z)},
\]
where $\chi$ denotes the Koszul-Vinberg characteristic function (\cite{vinberg1967theory, shima2007geometry, koszul1968deformations,koszul1961domaines,koszul1962ouverts, barbaresco2014koszul,faraut1994analysis,satake2014algebraic,vesentini1976invariant}):

\[
\chi(x,y,z) = 
\iiint\limits_{\substack{\zeta^2 > \xi^2 + \eta^2 \\[2pt] \zeta > 0}}
\exp\!\left[-(\xi x + \eta y + \zeta z)\right]
\, d\xi\, d\eta\, d\zeta.
\]

Let $q = z^2 - x^2 - y^2$.  
From \cite{shima2007geometry, faraut1994analysis}, we obtain the explicit expression
\[
\chi(x,y,z) = q^{-3/2} \chi(0,0,1),
\]
which yields the probability density

\begin{equation}
\label{thm:prob}
  P((x,y,z),(\xi,\eta,\zeta)) 
= \exp\left( -(x\xi + y\eta + z\zeta)
+ \frac{3}{2} \log q
- \log \chi(0,0,1) \right).  
\end{equation}

Thus, $\{P(\theta,\theta^*) : \theta \in \Theta\}$ forms an exponential family 
parametrized by $\Theta$. Its Fisher metric is given by
\[
g^{F}_{ij}(\theta) 
= \mathbb{E} \left[
\frac{\partial \log P}{\partial \theta^i}
\frac{\partial \log P}{\partial \theta^j}
\right], 
\quad \theta \in \Theta.
\]

A direct computation shows that
\[
g^{F} = -\frac{3}{2}\, (\nabla^{0})^{2} \log q,
\]
where $\nabla^0$ is the Levi--Civita connection of $\mathbb{R}^3$.
Hence, $(\Theta,\nabla^0,g^F)$ is a Hessian manifold and $g^F$ is a Koszul Hessian metric.
In coordinates, the metric is

\begin{equation*}
\begin{split}
g^F = \frac{3}{q^2} \Big[ &
(z^{2}+x^{2}+y^{2})\,dz^{2} - 4zx\,dz\,dx - 4zy\,dz\,dy \\
& + (z^{2}+x^{2}-y^{2})\,dx^{2} 
+ 4xy\,dx\,dy + (z^{2}-x^{2}+y^{2})\,dy^{2}
\Big].
\end{split}
\end{equation*}

Introducing cylindrical coordinates $(r,\alpha,t)$ by
\[
\begin{cases}
z = e^{t} \cosh r, \\
x = e^{t} \sinh r \cos \alpha, \\
y = e^{t} \sinh r \sin \alpha,
\end{cases}
\quad t \in \mathbb{R},\ r \ge 0,\ \alpha \in [0,2\pi),
\]
the metric simplifies to
\[
g = 3(dt^2 + dr^2 + \sinh^2 r \, d\alpha^2).
\]
Here $(r,\alpha)$ are polar coordinates on $\mathbb{H}^2$ and $t$ is the coordinate $\mathbb{R}$.
Thus, $(\mathbb{H}^2 \times \mathbb{R}, g)$ is isometric to $(\Theta, g^F)$. We deduce that, the leaf $L_\nabla$ of the associated foliation $\mathcal{F}_\nabla$ satisfies
\[
L_\nabla \cong \frac{\Theta}{\Lambda},
\qquad \Lambda \subset \mathrm{Isom}_0(\Theta).
\]

Identifying $\theta \in \Theta$ with the probability distribution $P_\theta$,
we obtain
\[
L_\nabla \cong 
 \frac{\{P_\theta : \theta \in \Theta\}}{\Lambda} \,
\qquad 
\Lambda \subset 
\mathrm{Isom}_0(\{P_\theta : \theta \in \Theta\}),
\]
where $\{P_\theta : \theta \in \Theta\}$ is an exponential family defined in \ref{thm:prob}.
\text{This completes the proof.}
\end{enumerate}    
\end{proof}
\begin{rem}
For the reader familiar with the  geometries of $3$-manifolds in the sense of Thurston, the leaf  
\[
L_\nabla \cong \frac{\mathbb{H}^2 \times \mathbb{R}}{\Gamma},
\qquad \Gamma \subset \mathrm{Isom}_0(\mathbb{H}^2 \times \mathbb{R}),
\]
is precisely the Seifert manifold with Euler number $0$ and hyperbolic orbifold base $B$
(since $\chi_{\mathrm{orb}}(B) < 0$). In other words, $L_\nabla$ admits 
a $\mathbb{H}^2 \times \mathbb{R}$-geometry. From \cite[Prop.~10.3.26]{martelli2016introduction}, $L_\nabla$ is covered by 
$\Sigma_g \times \mathbb{S}^1$ for some $g > 1$. Hence
\[
L_\nabla \cong \frac{\Sigma_g \times \mathbb{S}^1}{\Pi},
\]
where $\Pi$ is a finite group acting freely and properly on $\Sigma_g \times \mathbb{S}^1$. Since $\Sigma_g$ admits an $\mathbb{RP}^2$-structure 
(see \cite[p.~243]{benzecri1960varietes} and \cite[Cor.~6.5.2]{goldman2022geometric}), 
the universal cover of $\Sigma_g \times \mathbb{S}^1$, namely $\mathbb{H}^2 \times \mathbb{R}$, 
develops into the Lorentz cone $\Theta$. Consequently,
\[
\Sigma_g \times \mathbb{S}^1 \cong \frac{\Theta}{\Gamma \times \langle \lambda \rangle},
\qquad \Gamma \subset \mathrm{SO}(2,1),\quad \lambda > 1.
\]
We deduce that
\[ L_\nabla \cong \Pi \backslash \bigl( (\Gamma \times \langle \lambda \rangle) \backslash \Theta \bigr). \]

Finally, by using the Koszul-Vinberg characteristic function $\chi$ on the sharp convex cone $\Theta$,
and arguing as above, the result follows.
\end{rem}

\begin{disc}
 In this work, we focused primarily on the $3$-dimensional leaves of singular foliations, which we showed can be realized as quotients of exponential families under certain conditions. A natural open question remains: what occurs for leaves of higher dimensions? A deeper understanding of their geometric and topological structure could open new avenues in information geometry and statistics. Furthermore, we have mainly considered classical exponential families; extending our approach to other classes of statistical models represents a promising direction for future research.
\end{disc}

\bibliographystyle{splncs04}
\bibliography{main.bib}
\end{document}